\documentclass[12pt,reqno]{article}
mm \textheight=8.9in

\usepackage{amssymb}
\usepackage{amsmath}
\usepackage{amsthm}
\usepackage{pb-diagram}
\usepackage{color}
\newcommand{\xb}{\boldsymbol{x}}

\newcommand{\sms}{\hspace{.7pt}}

\newcommand{\br}[3]{{$#1$}$\lower4pt\hbox{$\tp\atop\raise4pt \hbox{$\scriptscriptstyle{#2}$}$} ${$#3$}}
\newcommand{\tw}[3]{{$#1$}${\,\scriptscriptstyle {#2}}\atop\raise9pt\hbox{$\scriptstyle\tp$} ${$#3$}}
\newcommand{\ttps}[2]{{#1}\raise5pt\hbox{$\lower12pt\hbox{$\scriptstyle\tp$}\atop \lower0pt\hbox{$\tilde\;$}$}\raise4.5pt\hbox{${\scriptstyle{#2}}$}}
\newcommand{\st}[1]{\mbox{${\,\scriptscriptstyle {#1}}\atop\raise5.5pt\hbox{$*$}$}}

\newcommand{\rd}[1]{\mbox{${\,\scriptscriptstyle {#1}}\atop\raise5.5pt\hbox{$\bullet$}$}}
\newcommand{\rt}[1]{\otimes_\chi}
\newcommand{\lt}[1]{\mbox{${\,\scriptscriptstyle {#1}}\atop\raise5.5pt\hbox{$\ltimes$}$}}
\newcommand{\btr}{\raise1.2pt\hbox{$\scriptstyle\blacktriangleright$}\hspace{2pt}}
\newcommand{\btl}{\raise1.2pt\hbox{$\scriptstyle\blacktriangleleft$}\hspace{2pt}}

\newcommand{\lcr}{\raise1.0pt \hbox{${\scriptstyle\rightharpoonup}$}}
\newcommand{\rcr}{\raise1.0pt \hbox{${\scriptstyle\leftharpoonup}$}}

\newcommand{\ttp}{{\lower12pt\hbox{$\tp$}\atop \hbox{$\tilde\;$}}}

\renewcommand{\deg}{\mathrm{deg}}
\newcommand{\gr}{\mathrm{gr}}

\newcommand{\id}{\mathrm{id}}

\newcommand{\Ru}{\mathcal{R}}

\newcommand{\Q}{\mathcal{Q}}
\renewcommand{\O}{\mathcal{O}}

\newcommand{\C}{\mathbb{C}}
\newcommand{\Z}{\mathbb{Z}}

\newcommand{\N}{\mathbb{N}}

\newcommand{\tp}{\otimes}

\newcommand{\U}{U}

\newcommand{\ve}{\varepsilon}
\newcommand{\gm}{\gamma}
\newcommand{\dt}{\delta}

\newcommand{\op}{\oplus}
\newcommand{\la}{\lambda}
\newcommand{\tr}{\triangleright}

\newcommand{\End}{\mathrm{End}}

\newcommand{\Span}{\mathrm{Span}}

\newcommand{\Hom}{\mathrm{Hom}}

\newcommand{\Tr}{\mathrm{Tr}}

\newcommand{\diag}{\mathrm{diag}}

\newcommand{\g}{\mathfrak{g}}
\renewcommand{\b}{\mathfrak{b}}

\newcommand{\h}{\mathfrak{h}}

\newcommand{\nb}{\boldsymbol{n}}

\newcommand{\s}{\mathfrak{s}}

\newcommand{\n}{\mathfrak{n}}

\newcommand{\nn}{\nonumber}
\newcommand{\p}{\mathfrak{p}}
\renewcommand{\l}{\mathfrak{l}}
\renewcommand{\c}{\mathfrak{c}}

\newcommand{\si}{\sigma}
\newcommand{\al}{\alpha}

\newcommand{\bt}{\beta}

\newcommand{\be}{\begin{eqnarray}}
\newcommand{\ee}{\end{eqnarray}}

\newtheorem{thm}{Theorem}[section]
\newtheorem{propn}[thm]{Proposition}
\newtheorem{lemma}[thm]{Lemma}
\newtheorem{corollary}[thm]{Corollary}

\newcount\prg

\newcommand{\parag}{\advance\prg by1 {\noindent\bf\thesection.\the\prg\hspace{6pt}}}

\begin{document}
\title{On representations of quantum conjugacy classes of $GL(n)$}
\author{
Thomas Ashton and Andrey Mudrov\footnote{This study is supported in part by the RFBR grant 09-01-00504.} \vspace{20pt}\\
\small Department of Mathematics,\\ \small University of Leicester, \\
\small University Road,
LE1 7RH Leicester, UK\\
}

\date{}
\maketitle

\begin{abstract}
Let $O$ be a closed Poisson conjugacy class of the complex algebraic Poisson group $GL(n)$ relative to the Drinfeld-Jimbo factorizable classical r-matrix. Denote by  $T$ the maximal torus of diagonal matrices
 in $GL(n)$.  With every $a\in O\cap T$ we associate a highest weight module $M_a$ over
 the quantum group $U_q\bigl(\g\l(n)\bigr)$ and an equivariant quantization  $\C_{\hbar,a}[O]$ of the polynomial ring $\C[O]$ realized by operators on $M_a$. All quantizations  $\C_{\hbar,a}[O]$ are isomorphic and can be regarded  as different exact representations of the same algebra, $\C_{\hbar}[O]$.
 Similar results are obtained for semisimple adjoint orbits in $\g\l(n)$ equipped with the canonical $GL(n)$-invariant Poisson structure.
\end{abstract}


\newpage
\section{Introduction}
Let $G$  denote the complex general linear algebraic group $GL(n)$ and let $\g$ be its Lie algebra $\g\l(n)$.
Regard $G$ as a Poisson group relative to the standard classical r-matrix and let $U_\hbar(\g)$ be the corresponding
quantum group. Consider a semisimple conjugacy class $O\subset G$, which is an affine subvariety of $G$.
An equivariant quantization of $O$
is a $\C[\![\hbar]\!]$-free deformation of the polynomial ring $\C[O]$ along with the $U(\g)$-action
to an associative algebra $\C_\hbar[O]$ and an action of $U_\hbar(\g)$. The indeterminate $\hbar$ is the deformation parameter  and $\C[\![\hbar]\!]$ is the ring of formal power series in $\hbar$ with complex coefficients. The algebra $\C_\hbar[O]$ was constructed in \cite{DM} and incorporated in
the general scheme in \cite{M1}.
In this paper, we present a family of exact representations of
$\C_\hbar[O]$ on $U_\hbar(\g)$-modules of highest weight. This family is parameterized
by diagonal matrices from $O$. Equivalently, with every diagonal matrix we associate
a highest weight module and an equivariant quantization of the conjugacy class of this matrix,
through an operator realization on that module. The quantized affine ring depends on $O$ and not on a
particular point in it. However, the modules are not isomorphic thus yielding non-equivalent exact representations of the same quantum conjugacy class.

Although the isotropy subgroups of all points in $O$ are isomorphic, not all are strictly compatible with the
standard triangular polarization of $\g$. We call
such a stabilizer a Levi subgroup if simple roots of its Lie algebra $\l$ are simple roots of $\g$, i.e. $\Pi^+_\l\subset \Pi^+_\g$. By this definition, $\l$ being a Levi subalgebra depends on a polarization of $\g$ relative
to a Cartan subalgebra, which is fixed once and for all. The quantization theory of the corresponding conjugacy class is standard: it can be realized by operators on a parabolic Verma module $M_\la$
relative to $U_q(\l)\subset U_q(\g)$.
General diagonal matrices in $O$ are uniquely parameterized by  Weyl group elements $\si$ satisfying
  $\si(R^+_\l)\subset R^+_\g$, where $R^+$ is the set of positive roots. For such $\si$ we construct a highest weight module $M_{\si.\la}$ and realize the algebra $\C_\hbar[O]$ in $\End(M_{\si.\la})$. Of course,
$M_{\si.\la}$ is a parabolic Verma module if  $\si (\Pi^+_\l)\subset \Pi^+_\g$.

A conjugacy class is simultaneously an adjoint orbit in $\End(\C^n)=\g$, and all the orbits in $\g$
are isomorphic to conjugacy classes. They are also
equipped with the canonical $U(\g)$-invariant Kirillov bracket restricted from $\g$.
The theory of quantum orbits in $\g$ and their representations is parallel to the theory of quantum conjugacy classes. It can be worked out by a straightforward rephrasing
of the key steps of this exposition. Another way to quantize semisimple orbits in $\g=\g\l(n)$ is through a two parameter
quantization, which can be obtained from quantum conjugacy classes by a formal shift trick and includes the quantization of Kirillov bracket as a limit case, \cite{DM}. In order to reduce the volume, we
do not focus on this case giving only the resulting formulas at the end of the paper.

An interesting feature of the non-parabolic quantization via $M_{\si.\la}$ is a lack of natural candidate for the quantum isotropy subgroup. This is true even in the case of Kirillov bracket on $O$. In this respect, this quantization
may help to understand the properties of quantum conjugacy classes which are
essentially non-Levi, that is, their isotropy subgroups are not isomorphic to Levi subgroups, \cite{M2,M3,M4}. Such classes
are not present in $GL(n)$ but form a large family in symplectic and orthogonal groups.

The paper is organized as follows. After the next preparatory section we look at representations
of the quantum upper and lower triangular unipotent groups. This is used for description of singular
vectors in the Verma modules and their tensor product with the natural representation of $U_q(\g)$.
We find the eigenvalues of a "quantum coordinate" matrix acting on the $U_q(\g)$-module
$\C^n\tp M_{\si.\la}$ and check that they are independent of $\si$. This enables us to construct
the representation of $\C_\hbar[O]$ in $\End(M_{\si.\la})$.

\section{Preliminaries}
It is an elementary fact from linear algebra that two  semisimple matrices are related by a conjugation
if and only if they have the same eigenvalues. So a conjugacy class $O$ is determined by the spectrum of
a matrix $A\in O$. This spectrum can be described by the  complex-valued vector  $\xb=(x_1, \ldots, x_k)$
of pairwise distinct eigenvalues and the integer-valued vector of multiplicities $\nb=(n_1,\ldots,n_k)$.
 All $x_i$ are invertible, while $n_i$ sum up to the dimension $n=n_1+\ldots+ n_k$.
The integer $k$ is assumed to be from the interval $[2,n]$, as the case $k=1$ is trivial.
The correspondence $(\xb,\nb)\mapsto O$ goes through the choice of
the initial point $o\in O$:
$$
o=\diag(\underbrace{x_1,\ldots, x_1}_{n_1},\ldots, \underbrace{x_k,\ldots, x_k}_{n_k}).
$$
The centralizer of $o$ in $G$ is the group $L=GL(n_1)\times\ldots \times GL(n_k)$, so $O$ is isomorphic
to $G/L$ as a $G$-space.
Note that the parametrization $(\xb,\nb)\mapsto O$ is not one-to-one, as a simultaneous permutation of $x_i$ and $n_i$
gives the same conjugacy class albeit a different initial point.

Restriction of $O$ to the maximal torus  $T$ of diagonal matrices is an orbit of the Weyl group, which
is the symmetric group $S_n$ in the case of study. It acts on diagonal matrices by permutation of entries,
$(\si o)_{ii}=o_{jj}$, $j=\si^{-1}(i)$, where $\si\in S_n$.
The isotropy subgroup of $o$ in $S_n$ is $S_{\nb} =S_{n_1}\times\ldots \times S_{n_k}$,
thus $O\cap T$ is in bijection with $S_n/S_{\nb}$.

The affine ring $\C[O]$ is the quotient
of the ring $\C[\End(\C^n)]$ by the ideal of relations
$$
(A-x_1)\ldots(A-x_k)=0,\quad \Tr(A^m)-\sum_{i=1}^k n_i x_i^m=0, \quad m=1,\ldots, k,
$$
where  $A=\sum_{i,j=1}^ne_{ij}\tp A_{ji}$ is the matrix of coordinate functions $A_{ji}$. Here $e_{ij}\in \End(\C^n)$ are the standard matrix units, $e_{ij}e_{lm}=\dt_{jl}e_{im}$. The left equality determines the vector $\xb$ while the values of $\Tr(A^m)$ fix the
vector $\nb$, up to a simultaneous permutation of their components.

The quantum conjugacy class $\C_\hbar[O]$ is described as follows.
Let $S\in \End(\C^n)\tp\End(\C^n)$ be the Hecke braid matrix associated with $U_q(\g)$, whose explicit form can be extracted from \cite{Ji}.
The quantized polynomial ring  $\C_\hbar[\End(\C^n)]$ is
generated over $\C[\![\hbar]\!]$ by the matrix entries $(K_{ij})_{i,j=1}^n$ subject to the
relations
\be
S_{12} K_2 S_{12}K_2= K_2 S_{12}K_2S_{12}
\label{RE}
\ee
written in the standard form of "reflection equation" in $\End(\C^n)\tp\End(\C^n)\tp\C_\hbar[\End(\C^n)]$.
The algebra $\C_\hbar[O]$
is a quotient of $\C_\hbar[\End(\C^n)]$ by the ideal of relations
$$
\prod_{i=1}^k(K-x_i)=0, \quad
\Tr_q(K^m)=\sum_{i=1}^kx^m_i[n_i]_q\prod_{j=1\atop j\not =i}^k\frac{q^{n_j}x_i-x_jq^{-n_j}}{x_i-x_j},\quad m=1,\ldots,k,
$$
with the q-trace of a matrix $X$  defined as $\Tr_q(X)=\sum_{i=1}^n q^{n+1-2i}X_{ii}$.
Here and further on,   $[z]_q=\frac{q^z-q^{-z}}{q-q^{-1}}$ for any $z\in \C$.

Let $R$ be the root system of $\g\l(n)$, $R^+$ the subset of positive roots and $\Pi^+=\{\al_i\}_{i=1}^{n-1}\subset R^+$ a basis of simple roots. Further we deal with the root systems of reductive  subalgebras $\l\subset \g$.
We label them with the subscript $\l$ reserving by default the notation $R=R_\g$ {\em etc}.
We use the standard relization of $R$ in a complex Euclidean vector space $\C^n$ with the inner product $(.,.)$, where
the simple positive roots are expressed in  an orthogonal basis $\{\ve_i\}_{i=1}^n$ by
$\al_i=\ve_i-\ve_{i+1}$, $i=1,\ldots, n-1$. This embedding identifies $\C^n$ with the dual vector space to the
Cartan subalgebra $\h\subset \g$.
We denote by  $h_\mu$ the image of $\mu$  under the isomorphism $\h^*\to \h$ implemented by the inner product: $\nu(h_\mu)=(\nu,\mu)$.

The quantum group $\U_\hbar(\g)$ is a $\C[\![\hbar]\!]$-algebra generated by $\h$ and $e_\al,f_\al$, $\al\in \Pi^+$,  subject to the relations, \cite{D},
$$
[h,e_\bt]=\bt(h)e_\bt,\quad
[h,f_\bt]=-\bt(h)f_\bt,\quad
[e_\al,f_\bt]=\dt_{\al,\bt}\frac{q^{h_\al}-q^{-h_\al}}{q-q^{-1}},\quad h\in \h, \quad\al,\bt \in \h^*,
$$
$$
e^2_\al e_\bt-[2]_qe_\al e_\bt e_\al+e_\bt e^2_\al=0
=
f^2_\al f_\bt-[2]_qf_\al f_\bt f_\al+f_\bt f^2_\al
$$
for $(\al,\bt)=-1$ and $[e_\al, e_\bt]=0=[f_\al, f_\bt]$ for $(\al,\bt)=0$.

The Hopf algebra structure on $\U_\hbar(\g)$ is defined through the comultiplication
$$
\Delta(e_\al)=e_\al\tp 1+q^{h_\al}\tp e_\al,\quad
\Delta(h_\al)=h_\al\tp 1+1\tp h_\al,\quad
\Delta(f_\al)=1\tp f_\al+f_\al\tp q^{-h_\al},
$$
counit $\epsilon(h_\al)=\epsilon(e_\al)=\epsilon(f_\al)=0$, and the antipode
$\gm(h_\al)=-h_\al$, $\gm(e_\al)=-q^{-h_\al}e_\al$, $\gm(f_\al)=-f_\al q^{h_\al}$. We use the Sweedler notation
$\Delta(x)=x^{(1)}\tp x^{(2)}$ for $x\in U_\hbar(\g)$.

The natural representation $\pi\colon U_\hbar(\g)\to \End(\C^n)$ is determined by the assignment
$\pi(h_{\al_i})=e_{ii}-e_{i+1\>i+1}$, $\pi(e_{\al_i})=e_{i\>i+1}$, $\pi(f_{\al_i})=e_{i+1\> i}$, for  $i\in [1,n)$.

We also work with  the quantum group $U_q(\g)$ as a $\C$-algebra assuming that $q$ is not a root of unit. It
is generated by  $\{q^{\pm h_{\ve_i}}\}_{i=1}^n$ and $\{f_\al, e_\al\}_{\al \in \Pi^+}$.
One can also consider $U_q(\g)$ over the ring $\C[q,q^{-1}]$ and its localizations. Further extension  over $\C[\![\hbar]\!]$
via $q=e^\hbar$ determines an embedding $U_q(\g)\varsubsetneq U_\hbar(\g)$, for which we use the same notation.
Until Proposition \ref{dynamical_roots_deformed}, $U_q(\g)$ is understood as a $\C$-algebra.

The quantum matrix space $\C_\hbar[\End(\C^n)]$ introduced in (\ref{RE}) is a $U_\hbar(\g)$-module algebra. The action is defined
on the generators by
$
(\id \tp x)(K)=\bigl(\pi\bigl(\gm(x^{(1)})\bigr)\tp \id \bigr)(K)\bigl(\pi(x^{(2)})\tp \id \bigr)$,
$x\in U_\hbar(\g)$, and extended to $\C_\hbar[\End(\C^n)]\ni a,b$ by the "quantum Leibnitz rule" $x(ab)=(x^{(1)}a)(x^{(2)}b)$.
There exists a homomorphism  $\C_\hbar[\End(\C^n)]\to U_\hbar(\g)$ implemented via the assignment
$K_{ij}\mapsto \Q_{ij}$, where $\Q$ is expressed through the universal R-matrix of $U_\hbar(\g)$
by $\Q=(\pi\tp \id)(\Ru_{21}\Ru)$. The image of this homomorphism is a quantization, $\C_\hbar[G]$,
of the coordinate ring of the group $G$. The algebra $\C_\hbar[O]$ is a quotient of $\C_\hbar[G]$.

\section{Natural representation of $U_q(\n_\pm)$.}
Consider the polarization $\g=\n_-\oplus \h\oplus \n_+$, where $\n_\pm$ are the nilpotent Lie subalgebras of positive and negative root subspaces. Let $\b_\pm=\h\oplus\n_\pm$ be the Borel subalgebras in $\g$.
Denote by $U_q(\n_\pm)$  the subalgebras in $U_q(\g)$ generated by $\{e_\al\}_{\al\in \Pi^+}$ and, respectively, $\{f_\al\}_{\al\in \Pi^+}$.
The quantum Borel subgroups $U_q(\b_\pm)$ are generated by $U_q(\n_\pm)$ over $U_q(\h)$; they are Hopf subalgebras in  $U_q(\g)$. The algebras $U_q(\n_\pm)$ and $U_q(\b_\pm)$ are deformations of the corresponding classical universal enveloping algebras. We consider a grading in $U_q(\b_\pm)$ with $\deg f_\al=1$, $\deg q^{\pm h_\al}=0$, $\al \in \Pi^+$.

Further we collect a few facts about the natural representation of $U_q(\g)$ on $\C^n$ and its restriction to
$U_q(\n_\pm)$.
Let $\{w_i\}_{i=1}^n\subset \C^n$ be the standard basis of columns with the only non-zero entry $1$ at the $i$-position from the top. The vector $w_i$ carries the weight $\ve_i$.
The natural representation of $U_q(\g)$ is determined by its restriction to the subalgebras
$U_q(\n_\pm)$, which is encoded in the diagrams
\begin{center}
\begin{picture}(200,40)
\put(0,0){$w_{n}$}
\put(10,15){\circle{3}}
\put(15,15){\vector(1,0){30}}
\put(50,15){\circle{3}}
\put(100,15){$\ldots$}
\put(55,15){\vector(1,0){30}}
\put(45,0){$w_{n-1}$}
\put(160,15){\circle{3}}
\put(125,15){\vector(1,0){30}}
\put(150,0){$w_{2}$}
\put(165,15){\vector(1,0){30}}
\put(200,15){\circle{3}}
\put(190,0){$w_{1}$}
\put(20,20){$e_{\al_{n-1}}$}
\put(60,20){$e_{\al_{n-2}}$}
\put(135,20){$e_{\al_{2}}$}
\put(175,20){$e_{\al_{1}}$}

\end{picture}
\qquad
\begin{picture}(200,40)
\put(0,0){$w_{n}$}
\put(10,15){\circle{3}}
\put(45,15){\vector(-1,0){30}}
\put(50,15){\circle{3}}
\put(100,15){$\ldots$}
\put(85,15){\vector(-1,0){30}}
\put(45,0){$w_{n-1}$}
\put(160,15){\circle{3}}
\put(155,15){\vector(-1,0){30}}
\put(150,0){$w_{2}$}
\put(195,15){\vector(-1,0){30}}
\put(200,15){\circle{3}}
\put(190,0){$w_{1}$}
\put(25,20){$f_{\al_{n-1}}$}
\put(65,20){$f_{\al_{n-2}}$}
\put(135,20){$f_{\al_{2}}$}
\put(175,20){$f_{\al_{1}}$}

\end{picture}
\end{center}
It follows from the diagrams that for every pair of integers $i,j\in [1,n]$ such that $i<j$ there is a unique Chevalley monomial $\psi_{ji}=f_{\al_{j-1}}\ldots f_{\al_{i}}\in U_q(\n_-)$ relating $w_i$ to $w_j$, that is $w_j=\psi_{ji} w_i$. Moreover, $\psi_{ji} w_l=\dt_{il}w_j$.

The algebras $U_q(\n_\pm)$ are isomorphic via the Chevalley involution $f_\al\to e_\al$.
We call contragredient the representation of $U_q(\n_\pm)$ on $\C^n$ given by $e_{k}w_{i}=\dt_{k,i}w_{i+1}$
and $f_{k}w_{i}=-\dt_{k,n-i+1}w_{i-1}$.
It factors through the automorphisms $e_i\mapsto e_{n-i}$, $f_i\mapsto f_{n-i}$, (inversion of Dynkin diagram) and the natural representation of $U_q(\b_\pm)$. Alternatively, it is a composition of the Chevalley involution $f_\al\leftrightarrow e_\al$ and natural representation.

For any finite dimensional $U_q(\g)$-module $W$ define the (right) dual representation on $W^*$
as $\langle w, x \tr u\rangle=\langle \gm^{-1}(x)\tr w,u\rangle$, where $w\in W$, $u\in W^*$, and $x\in U_q(\g)$.
Although $U_q(\n_\pm)$ are not Hopf algebras, their dual representations are still defined through
the embedding $U_q(\n_\pm)\subset U_q(\g)$.
Consider another copy of vector space $\C^n$ as dual to initial $\C^n$,
with the basis $\{v_i\}$, and the right conatural representation of $U_q(\n_+)$ on it.
Since $\gm^{-1}(e_\al)=-e_\al q^{-h_\al}$, one has
\be
e_\al w_i=\sum_{j=1}^n \pi(e_\al)_{ij}w_j, \quad e_\al v_i =-q^{-(\al, \ve_i)}\sum_{j=1}^n  \pi(e_\al)_{ji}v_j,
\label{quasi_dual}
\ee
where $\pi(e_\al)$, $\al\in \Pi^+$, are the matrices of the natural $\n_+$-action
on $\C^n$.

Consider the left ideal $J\subset U_q(\n_-)$ generated by $f_{\al_1}^2,f_{\al_i}$, $i>1$. Let
$N$ be the quotient module $U_q(\n_-)/J$. Remark that the automorphism of $U_q(\n_-)$
defined by $f_{\al_i}\mapsto a_if_{\al_i}$ for invertible $a_i\in \C$ leaves $J$ invariant and gives rise to
an automorphism of $N$.
\begin{propn}
$N$ is isomorphic to the natural $U_q(\n_-)$-module $\C^n$.
\end{propn}
\begin{proof}
This is a standard fact about finite-dimensional irreducible quotients of Verma  $U_q(\g)$-modules, \cite{Ja}.
The special case of $\C^n$ can be checked directly by constructing the obvious epimorphism $N\to \C^n$, $1\mapsto w_1$, and its section $w_1\mapsto 1+J$, $w_i\mapsto f_{\al_{i-1}}\ldots f_{\al_{1}}+J$, $i>1$.
\end{proof}

\begin{corollary}
The (right or left) conatural and contragredient representations of the algebras $U_q(\n_\pm)$ on $\C^n$ are isomorphic.
\end{corollary}
\begin{proof}
Indeed, in the case of $U_q(\n_+)$, they are cyclic representations  generated by the vector $w_1$ satisfying $e_{\al_1}^2 w_1=e_{\al_i} w_1=0$,
$i>2$. Hence they are quotients of $N$ by some submodules. Since their dimension is $n$, those submodules are zero, and the quotients are isomorphic to $N$. The case of $U_q(\n_-)$ is checked similarly.
\end{proof}
\begin{corollary}
\label{quot_nilp}
Let $V$ be a $U_q(\n_+)$-module and regard $\C^n$ as the conatural $U_q(\n_+)$-module. Then $\Hom_{U_q(\n_+)}(\C^n, V)= \{v\in V| e_{\al_1}^2v=e_{\al_i}v=0,\> i>1\}$.
\end{corollary}
\begin{proof}
Since the module $\C^n$ is cyclic, every homomorphism from $\Hom_{U_q(\n_+)}(\C^n, V)$ is determined by the assignment $1\mapsto v$, where the vector $v$ annihilates the ideal $J$.
\end{proof}
\section{Singular vectors.}
Denote by $\l=\g\l(n_1)\oplus \ldots \oplus \g\l(n_k)\subset \g\l(n)$ the stabilizer Lie algebra of the point $o\in O$.
Put $\p^\pm=\l+\n_\pm$ to be the parabolic subalgebras relative to $\l$.
The universal enveloping algebras $U(\l)$ and $U(\p^\pm)$
are quantized as Hopf subalgebras in  $U_q(\g)$.
So  $U_q(\l)$ is generated by $\{e_\al,f_\al\}_{\al\in \Pi^+_\l}$ over $U_q(\h)$ and $U_q(\p^\pm)$ are generated by $U_q(\n_\pm)$ over $U_q(\l)$.

For every $\la \in \frac{1}{\hbar}\h^*\oplus \h^*$ define a one-dimensional representation of $\U_q(\b_+)$ by
 $q^{\pm h_\al}\mapsto q^{\pm (\la,\al)}$,  $e_\al\mapsto 0$. Denoting it by $\C_\la$,
consider the Verma module $\hat M_\la=\U_q(\g)\tp_{\U_q(\b_+)}\C_\la$.
Let $\c_\l\subset \h$ be the center of $\l$ and $\c^*_\l\subset \h^*$ be the subset orthogonal to $\Pi^+_\l$.
Suppose that $\la\in \frac{1}{\hbar}\c_\l^*\oplus \c_\l^*$, so that $(\la,\al)=0$ for all $\al\in \Pi^+_\l$.
 For such $\la$, the $\U_q(\b_+)$-module  $\C_\la$ extends to a $\U_q(\p^+)$-representation, and  $\hat M_\la$ admits a projection
onto the parabolic Verma module $M_\la=\U_q(\g)\tp_{\U_q(\p^+)}\C_\la$. With
$x_i=q^{2(\la+\rho,\ve_{m_i})-2(\rho,\ve_1)}$, where $m_i=n_1+\ldots+n_{i-1}+1$,  $i=1,\ldots,k$,
the quantum conjugacy class $\C_\hbar[O]$ is realized by operators on $M_\la$.

Recall that a non-zero vector $v$ in a $U_q(\g)$-module is called singular if it generates  the trivial $U_q(\n_+)$-submodule, i.e. $e_\al v=0$, for all $\al \in \Pi^+$.
\begin{lemma}
\label{dual-contragredient}
Let $W$ be a finite dimensional $U_q(\g)$-module and $W^*$ its right dual module. Let $Y$ be a $U_q(\g)$-module. Singular vectors in $W\tp Y$
are parameterized by homomorphisms  $W^*\to Y$ of $U_q(\n_+)$-modules.
\end{lemma}
\begin{proof}
Choose a weight basis $\{w_i\}_{i=1}^d\subset W$,  where $d=\dim W$. Suppose that
$u\in W\tp Y$ is a singular vector,
$
u=\sum_{i=1}^d w_i\tp y_i,
$ for some
$y_i\in Y$. Let $\pi\colon U_q(\g)\to \End(W)$ denote the representation homomorphism.
We have, for $\al \in \Pi^+$,
$$
e_\al u=\sum_{i=1}^d \sum_{j=1}^d \pi(e_\al)_{ij}w_j\tp y_i+\sum_{i=1}^d q^{(\al,\nu_i)}w_i\tp e_\al y_i
$$
where $\nu_i$ is the weight of $w_i$.
So $e_\al  u=0$ is equivalent to
$
e_\al y_i=-q^{-(\al,\nu_i)}\sum_{j=1}^d  \pi(e_\al)_{ji}y_j.
$
The vector space $\Span\{y_i\}_{i=1}^d$ supports the right dual representation of $U_q(\n_+)$, provided
$y_i$ are linear independent. In general, it is a quotient of the right dual representation.
\end{proof}
\noindent
In particular, singular vectors in $Y\simeq \C\tp Y$ generate trivial $U_q(\n_+)$-modules, which recovers their definition.

Further we describe singular vectors  of certain weights in $\hat M_\la$ and $\C^n\tp \hat M_\la$. We need a few technical facts about $\hat M_\la$.
We define "dynamical root vectors" $\check{f}_{\al}\in U_q(\b_-)$  for all $\al \in R^+$.
For $\al \in \Pi^+$ we put $\check{f}_{\al}=f_\al$. If  $\al=\al_i+\ldots+\al_j=\al_i+\bt$,
we proceed recursively by
\be
\label{dyn_root_vectors}
\check{f}_{\al}=f_{\al_i}\check{f}_{\bt}\frac{q^{h_{\bt} + (\rho,\bt)}-q^{-h_{\bt} - (\rho,\bt)}}{q-q^{-1}}-\check{f}_{\bt}f_{\al_i}\frac{q^{h_{\bt} + (\rho,\bt)-1}-q^{-h_{\bt} - (\rho,\bt)+1}}{q-q^{-1}}.
\ee
Note that $q^{h_\bt}$ is well defined as an element of $U_q(\h)$ for $\bt\in \Z \Pi^+$.
The Cartan coefficients in $\check{f}_{\bt}$ commute with $f_{\al_i}$ and can be gathered on the right.
 By $\check{f}_\al(\la)$ we understand
an element from $U_q(\n_-)$ obtained through specialization of the coefficients at weight $\la$.
Clearly $\check{f}_\al(\la)v_\la=\check{f}_\al v_\la$.

Let $\g_{ln}\subset \g$ denote the subalgebra
$\g\l(n-l+1)$ with the root system $\{\al_{l},\ldots, \al_{n-1}\}$, $l=1,\ldots, n-1$.
The vectors $\check f_\al$ are generators of the Mickelsson algebras associated with
filtration $\g_{nn}\subset \ldots \subset \g_{1n}=\g$, \cite{Zh}. Their basic property is the equality
\be
e_{\al_j}\check{f}_{\al}^m v_\la=\dt_{ji}[m]_q [(\la+\rho,\al)-m]_q\check{f}_{\bt}\check{f}_{\al}^{m-1}v_\la.
\label{basic_dyn}
\ee
for any Verma module $\hat M_\la$ and any $m\in \N$, see e.g. \cite{M5}.
It is convenient to extend (\ref{dyn_root_vectors}) by $\check{f}_\al=1$ for $\al=0$ and $\check{f}_\al=0$ for
$\al\in -R^+$.  Then (\ref{basic_dyn}) is valid for all $\al>0$.

The following fact about $\check{f}_\al v_\la$ holds true. Its proof  can be found e.g. in \cite{M5}.
\begin{propn}
\label{singular}
Let $\al=\ve_i-\ve_j\in R^+$.
The vector $\check{f}_\al v_\la\in \hat M_\la $ is not vanishing at all $\la$.
It is  singular with respect to $U_q(\g_{i+1\sms n})$. Up to a scalar factor, it is a unique
$U_q(\g_{i+1\sms n})$-singular vector of weight $\la-\al$.
It is $U_q(\g)$-singular {\em iff} $q^{2(\la+\rho,\al)}=q^2$.
Up to a scalar factor, it is a unique singular vector of weight $\la-\al$.
\end{propn}

Further we apply Lemma \ref{dual-contragredient} to $W=\C^n$ and $Y=\hat M_\la$.
\begin{lemma}
\label{conat_n_Verma}
For all $l=1,\ldots,n$, there is a unique $U_q(\b_+)$-submodule in $\hat M_\la$ of lowest weight $\la+\ve_l-\ve_1$.
It is generated by $\check{f}_{\ve_1-\ve_l}v_\la$.
\end{lemma}
\begin{proof}
Proposition \ref{singular} implies that $\check{f}_{\ve_1-\ve_l}v_\la$ is a unique, up to a factor, $U_q(\g_{2\sms n})$-singular vector
 of this weight. It automatically satisfies the equation $e_{\al_1}^2\check{f}_{\ve_1-\ve_l}v_\la=0$ and generates
a unique $U_q(\b_+)$-submodule, which is a quotient of conatural, by Corollary \ref{quot_nilp}.
\end{proof}

Put $\la_i=(\la, \ve_i)\in \frac{1}{\hbar}\C$, $i=1,\ldots,n$, and present the singular vectors in $\C^n\tp \hat M_\la$ explicitly.
\begin{corollary}
\label{hat u_l}
Up to a scalar factor, the singular vector in $\C^n\tp \hat M_\la$  of weight $\la+\ve_l$, $l=1,\ldots,n$, is given by
$
\hat u_l=\sum_{i=1}^l (-q)^{i-1}\prod_{j=1}^{i-1}[\la_j-\la_l + l-j-1]_qw_i\tp \check{f}_{\ve_i-\ve_l}v_\la,
$
\end{corollary}
\begin{proof}
In accordance with Lemma \ref{conat_n_Verma}, put $y_1=\check{f}_{\ve_1-\ve_l}v_\la$ in $\hat u_l=\sum_{i=1}^n w_i\tp y_i$
and apply formula (\ref{basic_dyn}), $m=1$, to $y_{i+1}=-qe_{\al_i}y_i$ for $i>1$ taking into account $\check{f}_{\ve_i-\ve_l}=0$, $i>l$.
\end{proof}

\section{The $U_q(\g)$-module $\C^n\tp \hat M_\la$.}
We assume that $\la$ is an arbitrary weight from $\frac{1}{\hbar}\h^*\oplus \h^*$ unless specified otherwise.
Define $\hat V^\la_j\subset \C^n\tp \hat M_\la$ as a $\U_q(\n_-)$-submodule generated by $\{w_i\tp v_\la\}_{i=1}^j$.
It is also a $\U_q(\g)$-submodule, and the sequence $\hat V^\la_1\subset \ldots \subset \hat V^\la_n = \C^n\tp \hat M_\la$
forms a filtration. Its graded components $\hat V^\la_{j}/\hat V^\la_{j-1}$ are generated by the image of $w_j\tp v_\la$,
which is the highest weight vector. It is known that $\hat V^\la_{j}/\hat V^\la_{j-1}$ are isomorphic to
the Verma modules $\hat M_{\la+\ve_j}$  and determine the spectrum $\{\hat x_i\}_{i=1}^n$
of the  $\U_q(\g)$-invariant operator $\Q$, with $\hat x_i=q^{2(\la+\rho,\ve_i)-2(\rho,\ve_1)}=q^{2(\la,\ve_i)-2i+2}$, \cite{M1}.
The shifted  $S_n$-action  on $\h^*$
by $\si\colon\la\mapsto \si\cdot \la=\si(\la+\rho)-\rho$ permutes $\hat x_i$ to $\hat x_{\si^{-1}(i)}$.

Next we explain a diagram technique we use to study $\C^n\tp \hat M_\la$.
 We call $\psi_{ij}=f_{\al_{i}}\ldots f_{\al_{j-1}}\in U_q(\n_-)$ principal monomial of weight $\ve_j-\ve_i\in -R^+$. All other Chevalley monomials of this weight are obtained from $\psi_{ij}$ by permutation of factors. If $s\in S_{j-i}$
is such a permutation, then we write
$\psi_{ij}^s=f_{\al_{s(i)}}\ldots f_{\al_{s(j-1)}}=f^s_{\al_{i}}\ldots f^s_{\al_{j-1}}$.
We associate with $\psi_{ij}^s$
a horizontal graph with nodes $\{v^s_m\}_{m=i}^j\subset \hat M_\la$
setting $v^s_i=v_\la$, $v^s_{m}=f_{j+i-m}^s\ldots f_{j-1}^s v_\la$, $m=i+1,\ldots ,j$, and arrows
 $f_{j+i-m}^s\colon v^s_{m-1}\mapsto v^s_{m}$,
\begin{center}
\begin{picture}(200,35)
\put(0,0){$v^s_{j}$}
\put(10,15){\circle{3}}
\put(45,15){\vector(-1,0){30}}
\put(50,15){\circle{3}}
\put(100,15){$\ldots$}
\put(85,15){\vector(-1,0){30}}
\put(45,0){$v^s_{j-1}$}
\put(160,15){\circle{3}}
\put(155,15){\vector(-1,0){30}}
\put(150,0){$v^s_{i+1}$}
\put(195,15){\vector(-1,0){30}}
\put(200,15){\circle{3}}
\put(190,0){$v^s_{i}$}
\put(25,20){$f_{i}^s$}
\put(65,20){$f_{\al_{i+1}}^s$}
\put(135,20){$f_{j-2}^s$}
\put(175,20){$f_{j-1}^s$}

\end{picture}
\end{center}
In particular,  $v^s_{m}=v_{m}=\psi_{j+i-m,j}v_\la$ for $s=\id$. We combine it with the graph of natural representation oriented vertically to get the plane graph
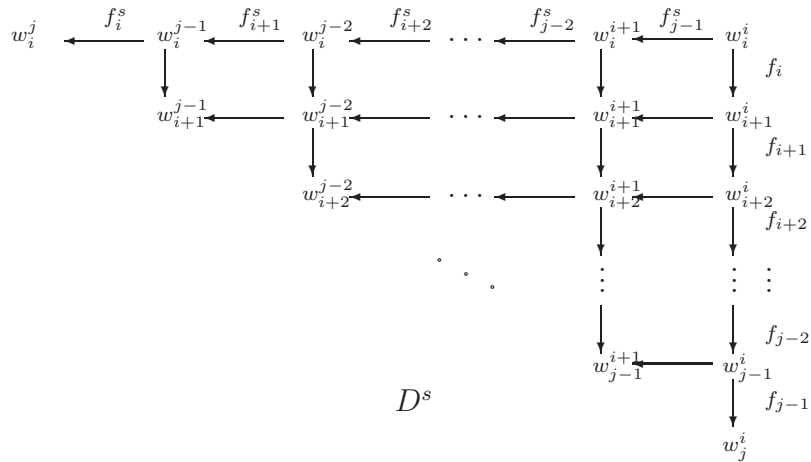
\begin{center}
\begin{picture}(300,180)
\put(150,20){$D^s$}
\put(5,160){$\scriptstyle{w^{j}_{i}}$}
\put(60,160){$\scriptstyle{w^{j-1}_{i}}$}
\put(115,160){$\scriptstyle{w^{j-2}_{i}}$}
\put(225,160){$\scriptstyle{w^{i+1}_i}$}
\put(275,160){$\scriptstyle{w^i_i}$}

\put(55,160 ){\vector(-1,0){30}}\put(108,160 ){\vector(-1,0){30}}\put(163,160 ){\vector(-1,0){30}}
\put(170,160 ){$\ldots$}
\put(218,160 ){\vector(-1,0){30}}
\put(270,161 ){\vector(-1,0){30}}

\put(40,167){$\scriptstyle{f^s_{i} }$}\put(91,167){$\scriptstyle{f^s_{i+1}}$}
\put(148, 167){$\scriptstyle{f^s_{i+2}}$}\put(201,167){$\scriptstyle{f^s_{j-2}}$}\put(250,167){$\scriptstyle{f^s_{j-1}}$}

\put(63,157 ){\vector(0,-1){18}}\put(119,157){\vector(0,-1){18}}\put(228,157 ){\vector(0,-1){18}}
\put(278,157 ){\vector(0,-1){18}}

\put(60,130){$\scriptstyle{w^{j-1}_{i+1}}$}
\put(115,130){$\scriptstyle{w^{j-2}_{i+1}}$}
\put(225,130){$\scriptstyle{w^{i+1}_{i+1}}$}
\put(275,130){$\scriptstyle{w^i_{i+1}}$}

\put(108,131 ){\vector(-1,0){30}}\put(163,131 ){\vector(-1,0){30}}
\put(170,131 ){$\ldots$}
\put(218,131 ){\vector(-1,0){30}}
\put(270,131 ){\vector(-1,0){30}}

\put(119,127){\vector(0,-1){18}}\put(228,127 ){\vector(0,-1){18}}
\put(278,127 ){\vector(0,-1){18}}

\put(115,100){$\scriptstyle{w^{j-2}_{i+2}}$}
\put(225,100){$\scriptstyle{w^{i+1}_{i+2}}$}
\put(275,100){$\scriptstyle{w^i_{i+2}}$}

\put(163,101 ){\vector(-1,0){30}}
\put(170,101 ){$\ldots$}
\put(218,101 ){\vector(-1,0){30}}
\put(270,101 ){\vector(-1,0){30}}

\put(228,97 ){\vector(0,-1){18}}
\put(278,97 ){\vector(0,-1){18}}
\put(227,65 ){$\vdots$}
\put(277,65 ){$\vdots$}

\put(187,67){\circle{1}}\put(177,72){\circle{1}}\put(167,77){\circle{1}}

\put(228,60 ){\vector(0,-1){18}}\put(278,60 ){\vector(0,-1){18}}
\put(278,32 ){\vector(0,-1){18}}

\put(225,35){$\scriptstyle{w^{i+1}_{j-1}}$}\put(274,35){$\scriptstyle{w_{j-1}^i}$}
\put(270,38 ){\vector(-1,0){30}}

\put(274,5){$\scriptstyle{w_{j}^i}$}

\put(290,22){$\scriptstyle{f_{j-1}}$}\put(290,47){$\scriptstyle{f_{j-2}}$}
\put(290,65 ){$\vdots$}
\put(290,90){$\scriptstyle{f_{i+2}}$}
\put(290,118){$\scriptstyle{f_{i+1}}$}\put(290,148){$\scriptstyle{f_i}$}

\end{picture}
\end{center}
where $w_l^{m}=w_l\tp v^s_m$.  By construction, $w_l^{m}$ depends on $s$, which dependence is suppressed in order to shorten the notation. The horizontal arrows designate the action of the corresponding operator
on the tensor factor $\hat M_\la$ while the vertical on the tensor factor  $\C^n$. If the operators assigned
to horizontal and vertical arrows applied to a node are distinct, the horizontal arrow amounts to the action
on the entire tensor product $\C^n\tp \hat M_\la$. In all cases, the comultiplied operator
associated with the horizontal arrow
maps $\C w_l^m\subset \C^n\tp \hat M_\la$ onto $\C w_l^{m+1}$ modulo $\C w_{l+1}^m$
(a direct consequence of the coproduct formula).
This immediately implies the following.
\begin{lemma}
\label{triangle}
Suppose a column $\{w_p^m\}_{p=r}^t$ for some $r<t$ lies in a $U_q(\g)$-submodule $M\subset \C^n\tp \hat M_\la$.
Then the column $\{w_p^{m+1}\}_{p=r}^{t-1}$ lies in $M$. If the generators assigned to vertical and horizontal
arrows with the origin $w_t^m$ are distinct, then $\{w_p^{m+1}\}_{p=r}^t$ lies in $M$.
\end{lemma}
In what follows we drop the symbol $s$ when it is clear from the context or arbitrary.
For $l\in [i,j]$ we call  $\Span \{w^{l+i-m}_m\}_{m=i}^{l}$ the $l$-diagonal.
For $l=j$ we call it  principal. Note that for $s=\id$ the principal
diagonal is spanned by $w_m^{j+i-m}=w_m\tp \psi_{mj}v$, $m=i,\ldots,j$.

Recall that the  monomial  $\psi_{ij}$ has weight $\ve_j-\ve_{i}$.
\begin{propn}
\label{principal_filtration}
If $s=\id$, then for all $l\in [i,j]$ the
tensors $w_l\tp \psi_{lj}v$ are proportional to each other modulo $\hat V^\la_{j-1}$:
$$
w_{l}\tp \psi_{lj}v=(-1)^{j-l}q^{\la_j-\la_l+l-j+1-\dt_{lj}} w_{j}\tp v \mod \hat V^\la_{j-1}.
$$
for all $l=i,\ldots, j-1$.
If $s\not=\id$, then $w_i\tp \psi_{ij}^s v\in \hat V^\la_{j-1}$.
\end{propn}
\begin{proof}
Denote by $D'$ the triangle in $D^s$ above the principal diagonal. By Lemma \ref{triangle},
it lies in $\hat V^\la_{j-1}$ as its right edge does. In particular, $w_l^{j+i-1-l}\in \hat V^\la_{j-1}$ for $l=i, \ldots,j-1$.

Suppose that $s=\id$. Observe  that the operator associated to horizontal and vertical arrows
directed from $w_l^{j+i-1-l}$ is the same and equal to $f_{\al_l}$. Applying $\Delta(f_{\al_l})$ to $w_l^{j+i-1-l}=w_l\tp \psi_{l+1,j}v_\la\in \hat V^\la_{j-1}$ we
get
$$
q^{(\ve_{l+1}-\ve_l,\ve_{j}-\ve_{l+1}+\la)}w_{l+1}^{j+i-1-l}+w_l^{j+i-l}=
q^{-(\ve_l-\ve_{l+1},\la)+(\ve_{l+1},\ve_j-\ve_{l+1})}w_{l+1}\tp \psi_{l+1,j}v_\la+w_l\tp \psi_{lj}v_\la,
$$
which belongs to $\hat V^\la_{j-1}$. This directly implies the first statement.

Suppose that $s\not =\id$.
Let $l$ be the rightmost integer $l\in [i,j)$ displaced by $s$, so that
$\psi_{ij}^s=f_{s(i)}\ldots f_{s(l)}f_{l+1}\ldots f_{j-1}$.
The node $w_l^{i+j-l-1}$ above the principal diagonal belongs to $D'$ and hence to $V^\la_{j-1}$.
Its vertical outward arrow is labeled with $f_l$ while horizontal with $f_{s(l)}\not =f_{l}$. Therefore
 $\Delta(f_{s(l)})$ acts on the column above $w_l^{i+j-l-1}$ inclusive strictly leftward and
maps it isomorphically onto the column rested on $w_{l+1}^{i+j-l-1}$ on the principal diagonal.
This column is the right bound of  a triangle that includes the part of principal diagonal from
$w_{l+1}^{i+j-l-1}$ up. It belongs to $V^\la_{j-1}$, hence the whole triangle including $w^i_j=w_i\tp \psi_{ij}^sv_\la$
belongs to $V^\la_{j-1}$.
\end{proof}

The initial point $o\in O$ determines a partition of the integer interval $[1,n]$ into the disjoint union
of $k$ subsets: $i,j$ are in the same subset if and only if $x_i=x_j$. This partition determines
a partial ordering on $[1,n]$: we write $i\prec j$ iff $i,j$ are from the same subset and $i<j$.
We call a permutation $\si\in S_n$ admissible if it
respects the ordering, i.e. $\si(i)<\si(j)$ once $i\prec j$.
\begin{lemma}
\label{admissible}
For every point $a\in O\cap T$ there is a unique admissible permutation $\si\in S_n$ such that $a=\si o$.
\end{lemma}
\begin{proof}
Indeed, if $i\prec j$ and $\si(i)>\si(j)$, the sign of inequality can be changed by combining $\si$ with the flip $(i,j)\in S_{\nb}$. This way, every permutation $\si$ such that $a=\si o$ can be adjusted so as to satisfy the required condition.
Uniqueness is obvious.
\end{proof}
\noindent
Lemma \ref{admissible} defines an embedding $S_n/S_{\nb}\subset S_n$ as a subset  of admissible permutations.
In terms of root systems,
$\si$ is admissible if and only if
$\si(R^+_\l)\subset R^+_\g$  or, equivalently, $\si(\pm R^+_\l)\subset \pm R^+_\g$ or, equivalently, $\si(\Pi^+_\l)\subset R^+_\g$.
Although the stabilizer of the point $\si(o)$ is isomorphic to $\l$, we call it a Levi subalgebra
{\em only if} $\si(\Pi^+_\l)\subset \Pi^+_\g$.

Let $\c^*_{\l,reg}$ denote the subset in $\c^*_{\l}$ such that for $\la\in \frac{1}{\hbar}\c^*_{\l,reg}$
the complex numbers $q^{2(\la,\ve_{m_i})}$, $i=1,\ldots,k$,  are pairwise distinct.
\begin{lemma}
\label{regular}
Suppose that $\la\in\frac{1}{\hbar}\c^*_{\l,reg}$ and $\si\in S_n/S_{\nb}$.
Let $\al \in \Pi^+_\l$ and $\si(\al)=\mu+\nu$ for some $\mu,\nu\in R_\g^+$. Then
$\bigl(\si(\la),\mu\bigr)\not=0\not=\bigl(\si(\la),\nu\bigr)$.
\end{lemma}
\begin{proof}
For any $\la\in \frac{1}{\hbar}\c^*_{\l,reg}$, the equality $\bigl(\si(\la),\mu\bigr)=0$ implies $\bigl(\la,\si^{-1}(\mu)\bigr)=0=\bigl(\la,\si^{-1}(\nu)\bigr)$. Therefore $\si^{-1}(\mu)$ and  $\si^{-1}(\nu)$ belong
to $R_\l$ and specifically to $R_\l^+$ since $\si$ is admissible. This contradicts the assumption that $\al$ is a simple root.
\end{proof}

\begin{propn}
\label{si-singular}
Suppose that $\la\in \frac{1}{\hbar}\c^*_\l\oplus \c^*_\l$ and $\si$ is an admissible permutation. For every $\al\in \Pi^+_\l$
the vector $v_{\la,\al}^\si=\check{f}_{\si(\al)} v_{\si\cdot \la}\in \hat M_{\si\cdot \la}$ is singular.
\end{propn}
\begin{proof}
As follows from (\ref{basic_dyn}) the vector
$v_{\la,\al}^\si$ is singular if and only if
$$
0=[(\si\cdot \la+\rho,\si\al)-1]_q=
[\bigl(\si(\la+\rho),\si(\al)\bigr)-1]_q=
[(\la+\rho,\al)-1]_q.
$$
This is the case for all $\al \in \Pi^+_\l$ since $\al=\ve_i-\ve_{i+1}$ for $i\prec i+1$ and $(\rho,\al)=1$.
\end{proof}

Introduce the subset $I_\l=\{m_i\}_{i=1}^k\subset [1,n]$
and  its complement $\bar I_\l$ in $[1,n]$.
Elements of $I_\l$ enumerate the highest weights of the irreducible $\l$-submodules in $\C^n$.
For a permutation $\si \in S_n/S_{\nb}$ put $I_\l^\si=\si (I_\l)$ and  $\bar I_\l^\si=\si (\bar I_\l)$.
Order the set $I^\si_\l=\{m_1^\si,\ldots, m_k^\si\}$ by $m_i^\si<m_{j}^\si$ for $i<j$. Note with care that $m_i^\si\not=\si(m_i)$.

Suppose that $\la\in \frac{1}{\hbar}\c^*_\l\oplus \c^*_\l$ and $\si$ is an admissible permutation. By Proposition
\ref{si-singular}, the vector $v^\si_{\la,\al}$ is
singular in $\hat M_{\si\cdot \la}$  for every $\al \in \Pi^+_\l$. Denote by $M_{\si\cdot \la}$ the $U_q(\g)$-module that is quotient
of $\hat M_{\si\cdot \la}$ by the submodule $\sum_{\al \in \Pi^+_\l}U_q(\g)v_{\la,\al}^\si$.
Let $\varpi$ be the projector $\hat M_{\si\cdot \la}\to  M_{\si\cdot \la}$.
Consider the filtration $(\hat V^{\si\cdot \la}_j)_{j=0}^n$ of $\C^n\tp \hat M_{\si\cdot \la}$ and put
 $V_i^{\si\cdot \la}=(\id\tp \varpi) (\hat V^{\si\cdot \la}_{m_i^\si})$, $i=1,\ldots,k$.

\begin{propn}
\label{levi_filtration}
For all $m\in \bar I^\si_\l$,  $(\id\tp \varpi)(\hat V^{\si\cdot \la}_{m}/\hat V^{\si\cdot \la}_{m-1})=\{0\}$.
The $\U_q(\g)$-modules $(V_i^{\si\cdot \la})_{i=1}^k$ form a filtration of $\C^n\tp M_{\si\cdot \la}$.
As a filtration of $\U_q(\n_-)$-modules, it is independent of $\la\in \frac{1}{\hbar}\c^*_\l\oplus \c^*_\l$
once $\si(\Pi^+_\l)\subset \Pi^+_\g$.
\end{propn}
\begin{proof}
For each  $m\in \bar I^\si_\l$ there is a positive integer $l<m$ such that $\al=\ve_l-\ve_m\in \si \Pi^+_\l$.
The vector $v_{\la,\al}^\si$ is singular in $\hat M_{\si\cdot \la}$ and vanishes in $M_{\si\cdot \la}$.
By Proposition \ref{principal_filtration},
$$
w_l\tp v_{\la,\al}^\si\simeq w_l\tp \psi_{lm} v_{\si\cdot \la}\simeq w_m\tp \hat v_{\si\cdot \la} \mod \hat V^{\si\cdot \la}_{m-1}.
$$
Projection to $\C^n\tp M_{\si\cdot \la}$ annihilates $w_l\tp v_{\la,\al}^\si$, hence $(\id\tp \varpi)(w_m\tp \hat v_{\si\cdot \la})\in (\id\tp \varpi)(\hat V^{\si\cdot \la}_{m-1})$. This proves the first and second statements.

Suppose that $\si(\l)$ is a Levi subalgebra, i.e. $\si(\Pi^+_\l)\subset \Pi^+_\g$.
Consider the subalgebra $U_q(\n'_-)$ in $U_q(\b_-)$ generated by $f_\al'=q^{h_\al}f_\al$, $\al\in \Pi^+$
(it is isomorphic to $U_q(\n_-)$). The $U_q(\n'_-)$-module $M_{\si\cdot \la}$ is
isomorphic to the  quotient of $U_q(\n'_-)$ by the left ideal $\sum_{\al \in \si\Pi^+_\l}U_q(\n'_-)f_{\al}$.
This ideal is independent of $\la$, hence $M_{\si\cdot \la}$ are isomorphic for all $\la$.
 With this isomorphism, the representation of $U_q(\n'_-)$ on $\C^n\tp M_{\si\cdot \la}$
is independent of $\la$ since $\Delta(f_\al')=f_\al'\tp 1+q^{h_\al}\tp f_\al'$. On the other hand, $V_{i}^{\si\cdot \la}/V_{i-1}^{\si\cdot \la}\simeq U_q(\b_-)(w_{m_i^\si}\tp v_{\si\cdot \la})=U_q(\n_-')(w_{m_i^\si}\tp v_{\si\cdot \la})$, $i=1,\ldots, k$ (here we identified $w_{m_i^\si}\tp v_{\si\cdot \la}$ with its image in $V_{i}^{\si\cdot \la}/V_{i-1}^{\si\cdot \la}$).
\end{proof}
\noindent
Proposition \ref{levi_filtration} gives an upper estimate $k$ for the degree of the minimal polynomial of $\Q$ on $\C^n\tp M_{\si\cdot \la}$. To make it exact, we must show that
all $V^{\si\cdot \la}_{i+1}/V^{\si\cdot \la}_i\not=\{0\}$. We do it in the next section.
\section{Realization of $\C_\hbar[O]$ in $\End(M_{\si\cdot \la})$}
Until Lemma \ref{vanishing_submodule}, $\la$ is an arbitrary weight from $\frac{1}{\hbar}\h^*\oplus \h^*$.
Define $\hat M_i^\la\subset \C^n\tp \hat M_\la$ to be the $U_q(\g)$-submodule generated by the singular vector $\hat u_i$ of weight $\la+\ve_i$, $i=1,\ldots,n$. The operator $\Q$ restricted to $\hat M_i^\la$ is scalar multiplication by $\hat x_i=q^{2(\la_i-i+1)}$.
For all $\si\in S_n$, the action $\si\colon \la\mapsto \si\cdot \la$ gives rise to the permutation
$\hat x_i\mapsto \hat x_{\si^{-1}(i)}$.

The contribution of principal monomials to the dynamical root vector gives
$$
\check{f}_{\ve_i-\ve_l}v_\la=
\prod_{j=i+1}^{l-1}[\la_j-\la_l + l-j]_q\psi_{il}v_\la+\ldots, \quad \la_i=(\la,\ve_i),
$$
where non-principal terms are omitted.
Applying Proposition \ref{principal_filtration}
we find that  $\hat u_l$ is equal to
$(-1)^{l-1}\hat C_lw_l\tp v_\la$ modulo $\hat V^\la_{l-1}$ with some scalar $\hat C_l$.
\begin{lemma}
For all $l=1,\ldots, n$, $\hat C_l=\prod_{j=1}^{l-1}
[\la_j-\la_l + l-j]_q.
$
\end{lemma}
\begin{proof}
For $l=2$, we have
$$
\hat u_2=-q[\la_1-\la_2]_qw_2\tp v_\la+w_1\tp f_1v_\la
=
-(q[\la_1-\la_2]_q+q^{-\la_1+\la_2})w_2\tp v_\la\mod \hat V^\la_1,
$$
so $\hat C_2=[\la_1-\la_2+1]_q$. For $l\geqslant 2$ we find
$$
\hat C_l=\sum_{i=1}^{l} q^{i-1}q^{\la_l-\la_i+i-l+1-\dt_{il}}
\prod_{j=1}^{i-1}[\la_j-\la_l + l-j-1]_q
\prod_{j=i+1}^{l-1}[\la_j-\la_l + l-j]_q.
$$
Suppose that the lemma is proved for $\hat C_{l-1}$. Then we can write
\be
(-1)^{l-1}\hat C_l&=&
\Bigl(\sum_{i=2}^{l} q^{i-2}q^{\la_l-\la_i+i-l+1-\dt_{il}}\prod_{j=2}^{i-1}
[\la_j-\la_l + l-j-1]_q\prod_{j=i+1}^{l-1}[\la_j-\la_l + l-j]_q\Bigr)
\times
\nn\\
&\times&
q[\la_1-\la_l + l-2]_q+q^{\la_l-\la_1-l+2}\prod_{j=2}^{l-1}[\la_j-\la_l + l-j]_q
\nn\\
&=&
\bigl(q[\la_1-\la_l + l-2]_q
+
q^{\la_l-\la_1-l+2}\bigr)\prod_{j=2}^{l-1}[\la_j-\la_l + l-j]_q.
\nn
\ee
We  applied the induction assumption to the expression in brackets in the top line.
The factor in the brackets is equal to $[\la_1-\la_l + l-1]_q$, so the statement is proved.
\end{proof}
Up to a non-zero factor,
$
\hat C_l=\prod_{j=1}^{l-1}(\hat x_j-\hat x_l),
$
where $\hat x_i$ are the eigenvalues of $\Q$ on $\C^n\tp \hat M_\la$.

\begin{lemma}
\label{M_j in M_i}
For $i<j$, the submodule $\hat M_{j}^\la$ is contained in $\hat M_i^\la$ if and only if $\hat x_i=\hat x_j$.
\end{lemma}
\begin{proof}
"Only if" is obvious. Suppose that $\hat x_i=\hat x_j$. Then the coefficient $\hat C_l$ turns zero
and $\hat u_l\in \hat V^\la_{l-1}$. First suppose that all $\hat x_l$ are pairwise distinct for $l<j$. Then
$\hat V^\la_{l-1}=\hat M_1^\la\oplus \ldots\oplus \hat M_{l-1}^\la$ and $\hat M_j^\la\subset  \hat M_i^\la$.
The vector $u_j$ is singular in $ \hat M_i^\la$, therefore $\hat u_j\simeq\check{f}_{\ve_i-\ve_j}\hat u_i$. This equality is true for
generic $\la$ subject to $\hat x_i=\hat x_j$, hence for all such $\la$.
\end{proof}
Define $\hat W^\la_i=\hat M_1^\la+\ldots+ \hat M_i^\la$, so that $\hat W^\la_i\subset \hat W^\la_j$, $i<j$.
\begin{propn}
\label{two_filtrations}
The submodules $\hat W^\la_i$ and $\hat V^\la_i$ coincide if and only if the eigenvalues $\{\hat x_l\}_{l=1}^i$ are
pairwise distinct.
\end{propn}
\begin{proof}
By Proposition \ref{principal_filtration}, $\hat W^\la_l\subset \hat V^\la_l$. If the eigenvalues are distinct, $\hat C_l\not =0$ for all $l$ and then $\hat V^\la_l\subset \hat W^\la_l$.
Otherwise $\hat M_j^\la\subset \hat M_i^\la$ for some $i<j$, by Proposition \ref{M_j in M_i}.
Then the graded modules $\gr \hat W^\la_n$ and $\gr \hat V^\la_n$ are different (recall that the latter is independent of $\la$ as a vector space, by Propostion \ref{levi_filtration} applied to $\l=\h$).
\end{proof}
Suppose that the weight $\la$ satisfies the condition $[(\la+\rho,\al)-1]_q=0$ for $\al\in R^+$
and let $\hat M_{\la-\al}\subset \hat M_\la$ be the submodule generated by  the singular vector $\check{f}_\al v_\la$.
Define the quotient module $M_{\la,\al}=\hat M_\la/\hat M_{\la-\al}$ and let $\varpi_{\al}\colon \hat M_\al \to M_{\la,\al}$ be the projector.
\begin{lemma}
\label{vanishing_submodule}
If $[(\la+\rho,\al)-1]_q=0$ for some $\al=\ve_i-\ve_j\in R^+$, then $(\id\tp \varpi_\al)(\hat M_j^\la)=\{0\}$.
\end{lemma}
\begin{proof}
Let us prove that $\hat u_j \in \C^n\tp \hat M_{\la-\al}\subset \C^n\tp\hat M_\la$. This is so if $i=1$ since $\hat u_j =w_1\tp \check{f}_{\al}v_\la$.
If $i>1$, the definition (\ref{dyn_root_vectors})
implies $\check{f}_{\ve_{i-1}-\ve_j}v_\la=[(\la+\rho,\al)]_q f_{\al_{i-1}}\check{f}_\al v_\la\in \hat M_{\la-\al}$. Proceeding by descending induction on $l$ one can check  that
$\check{f}_{\ve_{l}-\ve_j}v_\la \in \hat M_{\la-\al}$ for all $l\leqslant i$.
Indeed, all monomials constituting $\check{f}_{\ve_{l}-\ve_j}$ contain either the factor $f_{\al_{i-1}}f_\al $ or
$f_\al f_{\al_{i-1}}$, by (\ref{dyn_root_vectors}). The latter enters with the factor $[h_{\al} + (\rho,\al)-1]_q$, which can be pushed to the right and killed by $v_\la$. The vector $\check{f}_{\ve_{l}-\ve_j}$ is obtained from $\check{f}_{\ve_{i-1}-\ve_j}$ via generalized commutators
with $f_{\al_{m}}$, $m<i-1$, which commute with $\check{f}_{\al}$. This implies
$\hat u_j \in \C^n\tp \hat M_{\la-\al}$ and
$(\id\tp \varpi_\al)(\hat u_j)=0$, as required.
\end{proof}
\begin{corollary}
Let  $\la\in \frac{1}{\hbar}{\c^*_\l}$ and $\si\in S_n/S_{\nb}$. Then for any $j \in \bar I^\si_\l$ the
submodule $\hat M_j^\la$ is annihilated by the projection
$\id \tp \varpi\colon \C^n\tp \hat M_{\si\cdot \la}\to \C^n\tp  M_{\si\cdot \la}$.
\end{corollary}
\begin{proof}
Suppose that $j\in  \bar I^\si_\l$. There exists $i\in [1,n]$ such that $\al=\ve_i-\ve_j\in \si(\Pi^+_\l)$ and
$\check{f}_\al v_{\si\cdot \la}$ is singular in $\hat M_{\si\cdot \la}$. As follows from Lemma \ref{vanishing_submodule}, the submodule $\hat M_j^\la$ is annihilated
by the projection $\C^n\tp \hat M_{\si\cdot \la}\to \C^n\tp  M_{\si\cdot \la}$.
\end{proof}
\noindent
Now we turn to the submodules $\hat M_j^\la$ that survive under $\id \tp \varpi$; they are of  $j\in  I^\si_\l$.
\begin{lemma}
\label{u_reduction}
Suppose that  $\al=\ve_i-\ve_j$ and $\bt=\ve_l-\ve_m$  are such that $l\leqslant i<j< m$.
Suppose that $[(\la+\rho,\al)-1]_q=0$.
Then  the vector $\varpi_{\al}(\check{f}_{\bt}v_\la) \in M_{\la,\al}$ has a simple divisor $\hat x_j-\hat x_m$.
\end{lemma}
\begin{proof}
 We have the only condition on $\la$, which translates to the $\Q$-eigenvalues
as $\hat x_i=\hat x_j q^2$.  For almost all such $\la$ the coefficient
$\hat C_m\simeq \prod_{r=1}^{m-1}(\hat x_m-\hat x_r)$ is not zero, therefore
$(\id\tp \varpi_{\al})(\hat u_m)\in \C^n\tp M_{\la,\al}$ is not zero.
By Lemma \ref{M_j in M_i} we have $\hat M_m^\la\subset \hat M_j^\la$ at $\hat x_m=\hat x_j$.
Then $(\id \tp\varpi_{\al})(\hat M_m^\la)\subset (\id \tp\varpi_{\al})(\hat M_j^\la)=\{0\}$ and $(\id\tp \varpi_{\al})(\hat u_m)$
is divisible by $\hat x_m-\hat x_j$. The degree of this divisor is $1$, as it is simple in $\hat C_m$.
By Corollary \ref{hat u_l}, we can write
$$
\hat u_m=\sum_{l=1}^m (-q)^{l-1}\prod_{s=1}^{l-1}(\hat x_s-\hat x_mq^2)c_{ms}\>w_l\tp \check{f}_{\ve_l-\ve_m}v_\la, \quad \mbox{where all } c_{ms}\not=0.
$$
The part of the sum corresponding to $l>i$ is divisible by $\hat x_j-\hat x_m=\hat x_iq^{-2}-\hat x_m$.
Retaining the terms with $l\leqslant i$ we write
$
\hat u_m=\sum_{l=1}^i (-q)^{l-1}\prod_{s=1}^{l-1}(\hat x_s-\hat x_mq^2)c_{ms}\>w_l\tp \check{f}_{\ve_l-\ve_m}v_\la+
\ldots
$
Hence the vectors $\varpi_{\al}(\check{f}_{\ve_l-\ve_m}v_\la)$ are divisible by $\hat x_m-\hat x_j$ for all $l=1,\ldots,i$.
Clearly the degree of $\hat x_m-\hat x_j$ in $\varpi_{\al}(\check{f}_{\ve_1-\ve_m}v_\la)$ is $1$ since $\varpi_{\al}(\check{f}_{\ve_1-\ve_m}v_\la)$
generates the other coefficients in $(\id\tp \varpi_{\al})(\hat u_m)$. By (\ref{dyn_root_vectors}), $\hat x_m-\hat x_j$ is a divisor of degree $1$
in
$\varpi_{\al}(\check{f}_{\ve_l-\ve_m}v_\la)$ for all $l=1,\ldots, i$.
\end{proof}

Assuming $i=1,\ldots, k$ define the submodules $M_i^\la$, $W^\la_i$, and $V^\la_i$ in $\C^n\tp M_\la$  to be the images of $\hat M_{m_i}^\la$, $\hat W^\la_{m_i}$, and $\hat V^\la_{m_i}$  under the projection $\C^n\tp \hat M_\la\to \C^n\tp M_\la$. Define $x_i=\hat x_{m_i}$ to be the  eigenvalues of $\Q$ on $\C^n\tp M_\la$.
Put $x_i^\si=\hat x_{m_i^\si}$ and $M_i^{\si\cdot \la}\subset \C^n\tp  M_{\si\cdot \la}$
to be the image of $\hat M_{m_i^\si}^{\si\cdot \la}\subset \C^n\tp \hat  M_{\si\cdot \la}$.
\begin{propn}
The $U_q(\g)$-module  $\C^n\tp M_\la$ splits into the direct sum
$M_1^{\si\cdot \la}\op\ldots \op M_k^{\si\cdot \la}$ if and only if the eigenvalues $\{x_i\}_{i=1}^k$ are pairwise distinct.
\end{propn}
\begin{proof}
Put $C_i^\si=\prod_{j=1}^{i-1}(x_i^\si-x_j^\si)$ and
$\bar C_i^\si =\hat C_{m_i^\si}/C_i^\si\simeq\prod_{j\in [i,j)\cap  \bar I^\si_\l}(\hat x_{m^\si_i}-\hat x_j)$.
Define $u_i^\si=\frac{1}{\bar C_i^\si}(\id \tp \varpi)(\hat u_{m^\si_i})\in \C^n\tp M_{\si\cdot \la}$.
The module $M_i^{\si\cdot \la}$ is generated by the singular vector $u_i^\si$, which is a regular function of $\{x_{j}^\si\}$,
by Lemma \ref{u_reduction}. We have  $u_i^\si= C_i^\si w_{m^\si_i}\tp v_\la \mod V_{i-1}^{\si\cdot \la}$,
Therefore $V_i^{\si\cdot \la}=W_i^{\si\cdot \la}=M_1^{\si\cdot \la}\op\ldots \op M_i^{\si\cdot \la}$ if and only if $C_j^\si \not =0$ for all $j\in[1,i]$.
This immediately implies the assertion.
\end{proof}

Choosing $\la$ from $\frac{1}{\hbar}\c^*_{\l,reg}$ splits the set of eigenvalues $\{\hat x_i=q^{2(\la_i-i+1)}\}_{i=1}^n$ of $\Q$ to $k$ strings
$$
(x_1, x_1q^{-2}, \ldots, x_1q^{-2(n_1-1)}, \ldots,
x_k, x_kq^{-2}, \ldots,  x_kq^{-2(n_k-1)}),
$$
where $\hat x_i, \hat x_j$ enter a  string if and only if $i \preceq j$.
The lowest term in each string
is $x_i=q^{2(\la_{m_i}-m_i+1)}$, $i=1,\ldots,k$. They are exactly
 the eigenvalues of $\Q$  that survive in the projection $\C^n\tp \hat M_\la \to \C^n\tp M_\la$.
The matrix $\Q$ has the same eigenvalues on $\C^n\tp M_{\si\cdot \la}$ for all $\si$.
They are exactly the eigenvalues that survive in the projection $\C^n\tp \hat M_{\si\cdot \la} \to \C^n\tp M_{\si\cdot \la}$. The permutation $\si\in S_n$ induces a permutation
$(x_1,\ldots, x_k)\mapsto (x^\si_1,\ldots,x^\si_k)$.

At this point we turn to deformations and regard  $U_q(\g)$  as a $\C[\![\hbar]\!]$-subalgebra
of $U_\hbar(\g)$. Correspondingly, $U_q(\g)$-modules and their quotients are extended over  $\C[\![\hbar]\!]$
to become $U_\hbar(\g)$-modules.
The standard root vectors $f_\al\in U_\hbar(\n_-)$, $\al \in R^+$, generate a PBW-basis in $U_\hbar(\n_-)$, \cite{CP}. This basis establishes a $U_\hbar(\h)$-linear isomorphism $U_\hbar(\n_-)\simeq U(\n_-)\tp \C[\![\hbar]\!]$.
The root vectors $f_\al\in U_\hbar(\n_-)$, $\al \in \Pi^+$ are deformations of
their classical counterparts.
\begin{lemma}
\label{dynamical_roots_deformed}
Suppose that $\al \in \Pi^+_\l$ and $\si\in S_n/S_{\nb}$.
For any $\la\in \frac{1}{\hbar}\c^*_{\l,reg}$, the specialization $\check{f}_{\si(\al)}(\si\cdot \la)$ is a deformation of
$f_{\si(\al)}$,  upon a proper rescaling.
\end{lemma}
\begin{proof}
Let $i<j$ be the pair of integers such that $\si(\al) =\ve_i-\ve_j$. The statement is trivial
if $j-i=1$, since $\check{f}_{\si(\al)}=f_{\si(\al)}$ then. If $j-i>1$, then, by Lemma \ref{regular}, $q^{2\left(\si(\la),\ve_l\right)}\not =q^{2\left(\si(\la),\ve_j\right)}$ for all $l$ such that $i<l<j$.
Hence the modified commutators in $(q-q^{-1})^{j-i-1}\check{f}_{\si(\al)}(\si\cdot \la)$ are deformations of ordinary commutators, up to a non-zero multipliers.
On the other hand, the standard $f_{\si(\al)}$ is itself a composition of deformed commutators
of Chevalley genrators.
\end{proof}

\begin{propn}
Suppose that $\la\in \frac{1}{\hbar}\c^*_{\l,reg}$. Then the $U_q(\g)$-module $M_{\si\cdot \la}$ is $\C[\![\hbar]\!]$-free.
\end{propn}
\begin{proof}
The proof is similar to \cite{M3}, Proposition 6.2, where it is done for a certain quotient of the parabolic Verma module over
$U_q\bigl(\s\p(n)\bigr)$. It is based on a construction of a PBW basis in $\hat M_\la$, see therein, Section 6.
Here we indicate only the crucial point: for all $\al\in \Pi^+_\l$, $\si\in S_n/S_{\nb}$, and $\la\in \frac{1}{\hbar}\c^*_{\l,reg}$ the vectors  $\check{f}_{\si(\al)}v_{\si\cdot \la}$ can
be included in a PBW basis in $M_\la$ when the ring of scalars is $\C[\![\hbar]\!]$. This follows from Lemma \ref{dynamical_roots_deformed}.
\end{proof}

\begin{thm}
For all $\si\in S_n/S_{\nb}$ and $\la\in \frac{1}{\hbar}\c^*_{\l,reg}$ such that
$x_i=q^{2\la_i-2m_i+2}$, $i=1,\ldots, k$, the homomorphism of $\C_{\hbar}[\End(\C^n)]\to \End[M_{\si\cdot \la}]$ factors through an exact representation of $\C_\hbar[O]$.
\end{thm}
\begin{proof}
The minimal polynomial of $\Q$ and  $\Tr_q(\Q^m)$ are independent of $\si$, hence
the homomorphism of $\C_{\hbar}[G]\to \End[M_{\si\cdot \la}]$ factors through a homomorphism
$\C_\hbar[O]\to \End[M_{\si\cdot \la}]$. In the zero fiber,  the kernel of this homomorphism is zero.
Indeed, it is a proper invariant ideal in $\C[O]$ and it is zero since $O$ is a $G$-orbit. The algebra
$\C_\hbar[O]$ is $\C[\![\hbar]\!]$-free. It is a direct sum of isotypic $U_\hbar(\g)$-components, which
are finite over $\C[\![\hbar]\!]$, \cite{M1}. Therefore, the $U_\hbar(\g)$-invariant kernel is free. It is nil since
its zero fiber is nil; the representation of $\C_\hbar[O]$ in $\End[M_{\si\cdot \la}]$ is exact.
\end{proof}

A similar theory can be developed for quantization of the  Kirillov bracket on $\g^*$.
The quantum group $U_\hbar(\g)$ is replaced with $U(\g)\tp \C[\![t]\!]$, the algebra
$\C_\hbar[\End(\C^n)]$ with $U(\g_t)$ where
$
\g_t=\g[\![t]\!]
$
is a  $\C[\![t]\!]$-Lie algebra
with the commutator
$
[E_{ij},E_{lm}]_t=t\dt_{jl}E_{im}-t\dt_{im}E_{lj}
,$
$E_{ij}\in \g\subset \g[\![t]\!]$, $i,j=1,\ldots, n$.
The assignment $E_{ij}\mapsto t E_{ij}$ makes  $U(\g_t)$ a subalgebra in $U(\g)\tp\C[\![t]\!]$.
The dynamical root vectors are obtained from (\ref{dyn_root_vectors}) via the limit $q\to 1$.
The $U(\g)$-module $M_{\si\cdot \la}$ is defined similarly.
For $\la\in \frac{1}{t}\c^*_{l,reg}$, it is generated over $ \C(\!(t)\!)$ by the vector space $U(\n_-)v_{\si\cdot\la}$.
Its regular part $M_{\si\cdot \la}^{+}=U(\n_-)v_{\si\cdot\la}\tp \C[\![t]\!]$ is $U(\g_t)$-invariant.
The algebra $U(\g)$  acts on $\End(M_{\si\cdot \la}^{+})$
and on the image of $U(\g_t)$ in $\End(M_{\si\cdot \la}^{+})$.
The quantum orbit $\C_t[O]$ is described in terms of $E=\sum_{i,j=1}^ne_{ij}\tp E_{ji}\in \End(\C^n)\tp U(\g_t)$ as the quotient of $U_t[\g]$ by the ideal of
the relations (now $x_i$ may be not invertible but still pairwise distinct)
$$
\prod_{i=1}^k(E-x_i)=0, \quad
\Tr(E^m)=\sum_{i=1}^kx^m_in_i\prod_{j=1\atop j\not =i}^k\Bigl(1+\frac{t n_j}{x_i-x_j}\Bigr),\quad m=1,\ldots,k.
$$
These formulas can be obtained from a two parameter quantization at the limit $\hbar \to 0$. The two parameter
quantization can be formally obtained by a shift of the matrix $K$ and its eigenvalues, see \cite{DM} for details.
\begin{thm}
For all $\si\in S_n/S_{\nb}$ and $\la\in \frac{1}{t}\c^*_{\l,reg}$ such that
$x_i=2t(\la_i-m_i+1)$, $i=1,\ldots,k$, the homomorphism of $U(\g_t)\to \End(M_{\si\cdot \la}^+)$ factors through an exact representation of
the algebra $\C_t[O]$.
\end{thm}

\end{document}